\documentclass[12pt]{amsart}
\usepackage{amssymb,amsthm}
\usepackage{amsfonts,amscd}
\usepackage{graphicx}
\usepackage{xcolor}

\newtheorem{theorem}{Theorem}

\newtheorem{lemma}[theorem]{Lemma}

\theoremstyle{definition}

\newtheorem{remark}[theorem]{Remark}

\newcommand{\Sh}{\operatorname{Sh}}

\newcommand{\loc}{\operatorname{Loc}}
\newcommand{\rank}{\operatorname{rk}}

\title{A Conley index study of the evolution of the Lorenz strange
set}

\author{H\'ector Barge} 
\address{E.T.S. Ingenieros inform\'{a}ticos. Universidad Polit\'{e}cnica de Madrid. 28660 Madrid (Spain)}
\email{h.barge@upm.es}

\author{Jos\'e M.R. Sanjurjo}
\address{Facultad de C.C. Matem\'aticas. Universidad Complutense de Madrid. 28040 Madrid (Spain)}
\email{jose\_sanjurjo@mat.ucm.es}

\thanks{The authors are partially supported by the Spanish Ministerio de Econom{\'{i}}a y Competitividad (grant MTM2015-63612-P)}
\begin{document}

\begin{abstract}
In this paper we study the Lorenz equations using the
perspective of the Conley index theory. More specifically, we examine the
evolution of the strange set that these equations posses throughout the
different values of the parameter. We also analyze some natural Morse decompositions of the global attractor of
the system and the role of the strange set in these decompositions. We calculate the
corresponding Morse equations and study their change along the successive
bifurcations. We see how the main features of the evolution of the Lorenz system are
explained by properties of the dynamics of the global attractor. In addition, we formulate and prove some theorems which  are applicable in more
general situations. These theorems refer to Poincar\'{e}-Andronov-Hopf
bifurcations of arbitrary codimension, bifurcations with two homoclinic
loops and a study of the role of the travelling repellers in the
transformation of repeller-attractor pairs into attractor-repeller ones.
\end{abstract}

\maketitle


\section{Introduction}

Edward N. Lorenz studied in the 1960s a simplified model of fluid convection
dynamics in the atmosphere \cite{Lor}. This model is described by the following
family of differential equations, now known as the Lorenz equations,

\[
\begin{cases}
\frac{dx}{dt}=\sigma (y-x)\\
\frac{dy}{dt}=rx-y-xz\\
\frac{dz}{dt}=xy-bz
\end{cases},
\]

where $\sigma ,r$ and $b$ are three real positive parameters corresponding
respectively to the Prandtl number, the Rayleigh number and an adimensional
magnitude related to the region under consideration. As we vary the
parameters, we change the behaviour of the flow determined by the equations
in $\mathbb{R}^{3}$. The values $\sigma =10$ and $b=8/3$ have deserved
special attention in the literature. We shall fix them from now on, and we
shall consider the family of flows obtained when we vary the remaining
parameter, $r$.

Based on numerical studies of these equations, Lorenz found sensitive
dependence on initial conditions and he emphasized the importance of this
property in the study of natural phenomena, observing that, even in simple
models, trajectories are sensitive to small changes in the initial
conditions. He was able to prove that for every value of the parameter $r$
there is a bounded region (an ellipsoid) which every trajectory eventually
enters and never thereafter leaves. As a consequence, the existence of a
global attractor $\Omega $ of zero volume is established. This attractor is
the intersection of the successive images of the ellipsoid by the flow for
increasing times and should not be confused with the famous Lorenz
attractor, which is a proper subset of $\Omega $.

Afra\u{\i}movi\v{c}, Bykov and Sil'nikov \cite{AfBySi}, Williams \cite{Will} and Guckenheimer and Williams \cite{GuWi} constructed and studied a geometric model of the system based on the
numerically-observed features of the solutions of the Lorenz system. From
this model, the existence of a robust attractor can be derived. Tucker \cite{Tuc1, Tuc2} proved that, in fact, the Lorenz equations define a geometric Lorenz flow
and, as a consequence, they admit an attractor. In \cite{Lu} it is proved
that this attractor is mixing. Tucker's results were preceded by Mischaikow
and Mrozek \cite{MiMr1,MiMr2} and Mischaikow, Mrozek and Szymczak \cite{MiMr3} who gave a computer-assisted proof of the existence of chaos
in the Lorenz equations. An important contribution to the study of the
equations is the book \cite{Spa} by Sparrow. This book was written long before
Tucker's work was available and some of the global statements made in it are
only tentative. However, except for a few details, they have proved to agree
with Tucker's results. The topological classification of the Lorenz
attractors (for different parameter values) can be found in the paper \cite{Ra} by
D. Rand. A recent study of the global organization of the phase space in the
transition to chaos in the Lorenz system can be found in the recent paper \cite{DKO}
by Doedel, Krauskopf and Osinga (see also \cite{DKO1, DKO2}). See also \cite{BaSe2, BaSe} by R. Barrio and S. Serrano for related results. An additional reference is the book \cite{ArPa}, where
the elements of a general theory for flows on three-dimensional manifolds
are presented. The main motivation for this theory was, according to the
authors, the Lorenz equations. 

The present paper is devoted to the study of the Lorenz equations, using the
perspective of the Conley index theory. More specifically, we examine the
evolution of the strange set that these equations posses throughout the
different values of the parameter. We also analyze some natural Morse decompositions of the global attractor of
the system and the role of the strange set in these decompositions. We calculate the
corresponding Morse equations and study their change along the successive
bifurcations. We see how the main features of the evolution of the Lorenz system are
explained by properties of the dynamics of the global attractor. Particular importance is given to the evolution through the
preturbulent stage, just before the strange set becomes an attractor. It is proved that the transition from the preturbulent stage of the system
to the turbulent one is marked by a change in the internal dynamics of the
strange set, namely, an attractor-repeller splitting of the strange set,
which is present at the preturbulent stage, ceases to exist at the turbulent
stage. On the other hand, we see that the evolution of the system from the
homoclinic bifurcation to the Hopf bifurcation corresponds to a
transformation of a repeller-attractor decomposition of the global attractor
into an attractor-repeller one. This transformation is achieved via a
``travelling repeller". The purpose of the paper is to give a global vision from both the dynamical and
the topological perspectives and, based on the features of the Lorenz
equations, formulate and prove some theorems which reach well beyond the scope of these equations and are applicable in more
general situations. These theorems refer to Poincar\'{e}-Andronov-Hopf
bifurcations of arbitrary codimension, bifurcations with two homoclinic
loops and a study of the role of the travelling repellers in the
transformation of repeller-attractor pairs into attractor-repeller ones.

\section{Preliminaries}

Through the paper we  deal with families of flows $\varphi_\lambda:\mathbb{R}^n\times\mathbb{R}\to\mathbb{R}^n$ depending continously on a paremeter $\lambda\in[0,1]$. In some ocassions we assume that these families are induced by families of ODE's $\dot{X}=F_\lambda(X)$ depending differentiably on the parameter. In this case, it will be implicit that, for each $\lambda$, $F_\lambda$ is a $C^1$ map. 

\subsection*{Trajectories, Limit sets and stability.}
The main reference for the elementary concepts of dynamical systems will be \cite{BhSz} but we also recommend \cite{Rob1, PaMe, Pil, AlSaYo}.  

\vline

Let $\varphi:\mathbb{R}^n\times\mathbb{R}\to\mathbb{R}^n$ be a flow. Sometimes we write $xt$ instead of $\varphi(x,t)$ in order to simplify the notation.

\vline

We shall use the notation $\gamma(x)$ for the \emph{trajectory} of the point $x$, i.e. 
\[ 
 \gamma(x)=\{xt\mid t\in\mathbb{R}\}.
\] 
  Similarly for the \emph{positive semi-trajectory} and the \emph{negative semi-trajectory}
\[
\gamma^+(x)=\{xt\mid t\in\mathbb{R}^+\},\quad    \gamma^-(x)=\{xt\mid t\in\mathbb{R}^-\}.
\]  
 
By the \emph{omega-limit} of a point $x$ we understand the set 
\[
\omega(x)=\bigcap_{t>0}\overline{x[t,\infty)}.
\] 
In an analogous way, the \emph{negative omega-limit} is the set 
\[
\omega^*(x)=\bigcap_{t<0}\overline{x(-\infty,t]}.
\]

We recall that if $\omega(x)$ (resp. $\omega^*(x)$) is compact then it must be connected.

\vline

 \subsection*{Attractors} In the literature there are many different definitions of attractor as it has been pointed out by Milnor \cite{MilAt}. Among the definitions treated by Milnor we shall use that of an asymptotically stable compactum.  An invariant compactum $K$ is \emph{stable} if every neighborhood $U$ of $K$
contains a neighborhood $V$ of $K$ such that $V[0,\infty
)\subset U$. Similarly, $K$ is \textit{negatively stable} if every
neighborhood $U$ of $K$ contains a neighborhood $V$ of $K$ such that $V
(-\infty ,0]\subset U$. 

The compact invariant set $K$ is said to be \textit{%
attracting} provided that there exists a neighborhood $U$ of $K$ such that $\emptyset\neq\omega (x)\subset K$, for every $x\in U$, and \emph{repelling} if there
exists a neighborhood $U$ of $K$ such that $\emptyset\neq\omega ^{\ast }(x)\subset K$ for every  $x\in U$.

 An \textit{attractor }(or \emph{asymptotically stable} compactum)\textit{\ }is an attracting stable set and a \textit{repeller }is
a repelling\textit{\ }negatively stable set. We stress the fact that
stability (positive or negative) is required in the definition of attractor
or repeller. 

If $K$ is an attractor, its region (or basin) of attraction of $K$ is the set 
\[
 \mathcal{A}(K)=\{x\in M \mid \emptyset\neq\omega (x)\subset K\}.
\] 
It is well known, that $\mathcal{A}(K)$ is an open invariant set. If in particular $\mathcal{A}(K)$ is the whole phase space we say that $K$ is a \emph{global attractor}. 

\vline

\subsection*{Dissipative flows.} 
Let $M$ be a non-compact, locally compact metric space. A flow $\varphi:M\times\mathbb{R}\to M$ is \emph{dissipative} provided that for each $x\in M$, $\omega(x)\neq\emptyset$ and the closure of the set
\[
\Omega(\varphi)=\bigcup_{x\in M}\omega(x)
\] 
is compact.

The dissipativeness of $\varphi$ is equivalent to the existence of a global attractor or, equivalently, to $\{\infty\}$ being a repeller for the flow extended to the Alexandroff compactification of $M$, leaving $\infty$ fixed. This was proved by Pliss \cite{Pli}. Dissipative flows have been introduced by Levinson \cite{Lev}. An interesting reference regarding dissipative flows is \cite{Hal}.

\vline

\subsection*{Isolated invariant sets and isolating blocks.} An important class of invariant compacta is the so-called \emph{isolated invariant sets} (see \cite{Con,ConEast, EaJDE} for details). These are compact invariant sets $K$ which possess an \emph{isolating neighborhood}, i.e. a compact neighborhood $N$  such that $K$ is the maximal invariant set in $N$.  

 A special kind of isolating neighborhoods will be useful in the sequel, the so-called \emph{isolating blocks}, which have good topological properties. More precisely, an isolating block $N$ is an isolating neighborhood such that there are compact sets $N^i,N^o\subset\partial N$, called the entrance and the exit sets, satisfying
\begin{enumerate}
\item $\partial N=N^i\cup N^o$;
\item for each $x\in N^i$ there exists $\varepsilon>0$ such that $x[-\varepsilon,0)\subset M\setminus N$ and for each $x\in N^o$ there exists $\delta>0$ such that $x(0,\delta]\subset M\setminus N$;
\item for each $x\in\partial N\setminus N^i$ there exists $\varepsilon>0$ such that $x[-\varepsilon,0)\subset \mathring{N}$ and for every $x\in\partial N\setminus N^o$ there exists $\delta>0$ such that $x(0,\delta]\subset\mathring{N}$.
\end{enumerate} 

These blocks form a neighborhood basis of $K$ in $M$. If the flow is differentiable, the isolating blocks can be chosen to be
manifolds which contain $N^{i}$ and $N^{o}$ as submanifolds
of their boundaries and such that $\partial N^{i}=\partial N^{o}=N^{i}\cap
N^{o}$.

\vline

\subsection*{Hartman-Grobman blocks and complex invariant manifolds}

 Let $\dot{X}=F(X)$ an ODE defined in $\mathbb{R}^n$ and let $\varphi$ be the (local) flow induced by this vector field. Suppose that $\dot{X}=F(X)$ possesses a hyperbolic fixed point $p$ and let $\varphi_*$ the flow induced  by the linearization $\dot{Y}=dF(p)Y$ of $\dot{X}=F(X)$. Then, Hartman-Grobman Theorem (see \cite[Chapter~2, pg. 120]{Perk} or \cite[Theorem~II.3, pg. 53]{Zehn}) ensures that there exist neighborhoods $U$ of $p$ and $V$ of $0$ in $\mathbb{R}^n$ and a homeomorphism $h: U\to V$ such that $h(\varphi(x,t))=\varphi_*(h(x),t)$ if $\varphi(x,[0,t])\subset U$.  Let $\delta>0$ be such that the closed ball $\overline{B_\delta(0)}$ of radius $\delta$ centered at $0$ is contained in $V$. Notice that $\overline{B_\delta(0)}$ is an isolating block of $\{0\}$ for $\varphi_*$ if the norm of $\mathbb{R}^n$ is chosen properly. Then, it follows that $h^{-1}(\overline{B_\delta(0)})$ is an isolating block of $\{p\}$ for $\varphi$. We shall call this kind of blocks \emph{Hartman-Grobman blocks} of the hyperbolic fixed point $\{p\}$. 
 
 Consider an ODE $\dot{X}=F(X)$ defined in $\mathbb{R}^3$ and let $p$ be a hyperbolic fixed point having one negative eigenvalue $\beta$ and two complex conjugated eigenvalues $\mu\pm\nu i$ with $\mu<0$.  Let $\dot{Y}=dF(p)Y$  be the linearization of $\dot{X}=F(X)$, $E$ the invariant $2$-dimensional subspace associated to the complex eigenvalues and $\delta>0$ such that $\overline{B_\delta(0)}\subset V$. We call \emph{local complex invariant manifold} of $p$ to the positively invariant open $2$-disk $\loc W^\mathbb{C}=h^{-1}(E\cap B_\delta(0))$. The \emph{complex invariant manifold} of $p$ is defined as the set of points whose forward trajectory eventually enters $\loc W^\mathbb{C}$, i.e.
\[
W^\mathbb{C}=\{x\in\mathbb{R}^3\mid  xt\in \loc W^\mathbb{C}\; \mbox{for some}\; t>0\}.
\]

\vline

\subsection*{Algebraic topology and shape theory.}
We use some topological notions through this paper. We recommend the books of Hatcher and Spanier \cite{Hat,Span} to cover this material. We use the notation $H^*$ for \v Cech cohomology. We consider cohomology taking coefficients in $\mathbb{Z}$. We recall that \v Cech and singular cohomology theories agree on polyhedra and manifolds and, more generally, on ANRs. 

If a pair of spaces $(X,A)$ satisfies that its cohomology $H^k(X,A)$ is finitely generated for each $k$ and is non-zero only for a finite number of values of $k$ (as it happens if $(X,A)$ is a pair of compact manifolds), its \emph{Poincar\'e polinomial} is defined as
\[
P_t(X,A)=\sum_{k\geq 0} \rank H^k(X,A)t^k.
\]

There is a form of homotopy which has proved to be the most convenient for the study of the global {to\-po\-lo\-gi\-cal} properties of the invariant spaces involved in dynamics, namely \emph{Borsuk's homotopy theory} or \textit{shape theory}, introduced and studied by Karol Borsuk.  We are not going to make a deep use of shape theory but we recommend to the interested reader the books \cite{Bormono,MarSe, DySe} for an exhaustive treatment of the subject, and the papers \cite{KapRod, GMRSNLA, SanMul, GMRSo, SGRAC, Hast, Sanin, Rob2} for a short comprehensive introduction and some applications to dynamical systems. We only recall that shape theory and homotopy theory agree when dealing with manifolds, CW-complexes or, more generally, ANRs and that \v Cech cohomology is a shape invariant.

\vline

\subsection*{Conley index.}
Let $K$ be an isolated invariant set. Its \emph{Conley index} $h(K)$ is defined as the pointed homotopy type of the topological space $(N/N^o,[N^o])$, where $N$ is an isolating block of $K$.  A weak version of the Conley index which will be useful for us is the \emph{cohomological index} defined as $CH^*(K)= H^*(h(K))$. It can be proved that $CH^*(K)\cong H^*(N, N^o)$. Our main references for the Conley index theory and its applications are \cite{Con, ConZehnInv, ConZehn, Sal, RobSal}.  In addition, some applications of the Conley index theory to the study of the Lorenz equations can be seen in \cite{Sanjnon, Sanjhopf, GSIAM}.

\vline

\subsection*{Morse decompositions and equations.}

We recall that if $K$ is a compact invariant set, the finite collection  $\{M_1,\ldots,M_n\}$ of pairwise disjoint invariant subcompacta of $K$ is a \emph{Morse decomposition} if it satisfies that
\[
\mbox{for each}\; x\in\left(K\setminus\bigcup_{i=1}^n M_i\right),\quad \omega(x)\subset M_j\;\mbox{and}\; \omega^*(x)\subset M_k\;\mbox{with}\; j<k.   
\]
Each set $M_i$ is said to be a \emph{Morse set}.

Given a Morse decomposition $\{M_1,M_2,\ldots, M_k\}$ of an isolated invariant set $K$, there exists a polynomial $Q(t)$ whose coefficients are non-negative integers such that
\[
\sum_{i=1}^n P_t(h(M_i))=P_t(h(K))+(1+t)Q(t).
\]

This formula, which relates the Conley indices of the Morse sets with the Conley index of the isolated invariant set is known as the \emph{Morse equations} of the Morse decomposition and it generalizes the classical Morse inequalities.

\vline

\subsection*{Hausdorff distance}

Let $X$ be a complete metric space with metric $d$ and consider the \emph{hyperspace} $\mathcal{H}(X)$ whose elements are the non-empty subcompacta of $X$.  We recall that the \emph{Hausdorff distance} in $\mathcal{H}(X)$ is defined as
\[
d_H(A,B)=\inf\{\varepsilon>0\mid B\subset A_\varepsilon\;\mbox{and}\; A\subset B_\varepsilon\},
\]

where $A_\varepsilon$ and $B_\varepsilon$ denote the open $\varepsilon$-neighborhoods of $A$ and $B$, with respect to the metric $d$, respectively. For more information about the Hausdorff distance and its properties we recommend de book \cite{Barn}.

\subsection*{Continuations of isolated invariant sets.}
In this paper the concept of continuation of isolated invariant sets plays a crucial role. Let $M$ be a locally compact metric space, and let $\varphi_\lambda:M\times\mathbb{R}\to  M$ be a parametrized family of flows (parametrized by $\lambda\in[0,1]$, the unit interval). The family $(K_\lambda)_{\lambda\in J}$,  where $J\subset[0,1]$ is a closed (non-degenerate) subinterval and, for each $\lambda\in J$, $K_\lambda$ is an isolated invariant set for $\varphi_\lambda$ is said to be a \emph{continuation} if for each $\lambda_0\in J$ and each $N_{\lambda_0}$ isolating neighborhood for $K_{\lambda_0}$, there exists $\delta>0$ such that $N_{\lambda_0}$ is an isolating neighborhood for $K_\lambda$ for every $\lambda\in (\lambda_0 -\delta, \lambda_0 + \delta)\cap J$. We say that the family $(K_\lambda)_{\lambda\in J}$ is a continuation of $K_{\lambda_0}$ for each $\lambda_0\in J$.

Notice that \cite[Lemma~6.1]{Sal} ensures that if $K_{\lambda_0}$ is an isolated invariant set for $\varphi_{\lambda_0}$, there always exists  a continuation $(K_\lambda)_{\lambda\in J_{\lambda_0}}$ of $K_{\lambda_0}$ for some closed (non-degenerate) subinterval $\lambda_0\in J_{\lambda_0}\subset[0,1]$.

There is a simpler definition of continuation based on \cite[Lemma 6.2]{Sal}. There, it is proved that if $\varphi_\lambda : M \times\mathbb{R}\to M$ is a parametrized family of flows and if $N_1$ and $N_2$ are isolating neighborhoods of the same isolated invariant set for $\varphi_{\lambda_0}$, then there exists $\delta>0$ such that $N_1$ and $N_2$ are isolating neighborhoods for $\varphi_\lambda$, for every $\lambda\in(\lambda_0-\delta,\lambda_0 +\delta)\cap[0,1]$, with the property that, for every $\lambda$, the isolated invariant subsets in $N_1$ and $N_2$ which have $N_1$ and $N_2$ as isolating neighborhoods coincide.

Therefore, the family $(K_\lambda)_{\lambda\in J}$, with $K_\lambda$ an isolated invariant set for $\varphi_\lambda$, is a continuation if for every $\lambda_0\in J$ there are an isolating neighborhood $N_{\lambda_0}$ for $K_{\lambda_0}$ and a $\delta > 0$ such that $N_{\lambda_0}$ is an isolating neighborhood for $K_\lambda$, for every $\lambda\in(\lambda_0-\delta,\lambda_0 +\delta)\cap J$.

Notice that, since this should not lead to any confusion, sometimes we will only say that $K_\lambda$ is a continuation of $K_{\lambda_0}$ without specifying the subinterval $J\subset[0,1]$ to which the parameters belong.

We shall make use of \cite[Theorem~4]{sanjuni} which states that if $K_{\lambda_0}$ is an atractor for $\varphi_{\lambda_0}$ and $(K_\lambda)_{\lambda\in J}$ is a continuation of $K_{\lambda_0}$, then there exists $\delta>0$ such that $K_\lambda$ is an attractor of the same shape of $K_{\lambda_0}$ for $\lambda\in(\lambda_0-\delta,\lambda_0+\delta)\cap J$.
\section{Generalized Pitchfork bifurcations}

We shall use along the paper some facts about the Lorenz equations which can
be found in the existing literature. We recommend, in particular, the book
by Sparrow \cite{Spa}.

For $r\leq 1$ the origin is a global attractor (this includes $r=1$ although
for $r=1$ there are two negative eigenvalues and the third is equal to
zero). For $r>1$ there are two additional singularities $C_{1}$ and $C_{2}$
which are attractors until $r=24.74$ (when a Hopf bifurcation takes place).
 For $r>1$ the origin is a hyperbolic fixed point with a two-dimensional stable manifold and a
one-dimensional unstable manifold. All the points not lying in $W^{s}(0)$
are attracted by $C_{1}$ or $C_{2}$ until the value $r=13.926$, when a
homoclinic bifurcation takes place. For all $r$ with $1<r<13.926$, the
unstable manifold of $0$ consists of two orbits attracted by $C_{1}$ and $%
C_{2}$ respectively, together with $0$. Hence, at $r=1$, a pitchfork
bifurcation takes place in the origin, which is an attractor for $r=1$ and
becomes a hyperbolic non stable fixed point for $r>1$. We summarize the
discussion in the following statement:

\subsection*{-}\textbf{For }$r\leq 1$ \textbf{the origin is a global attractor and for }$%
r>1 $ \textbf{it becomes a hyperbolic fixed point with a two-dimensional
stable manifold and a one-dimensional unstable manifold. }

\bigskip

This is a particular example of a phenomenon which can be studied in a more
general form in $\mathbb{R}^{n}$ with an arbitrary distribution of positive
and negative eigenvalues. There are two extreme cases: a) when the
origin becomes a hyperbolic point with dimension of $W^{u}(0)$ equal to $1$
(which is the current situation with $n=3$) and b) when the origin becomes a
hyperbolic point with dimension of $W^{u}(0)$ equal to $n$ or, in other
words, the origin becomes a repeller. The second case has been called a%
\textit{\ generalized Poincar\'{e}-Andronov-Hopf bifurcation} \cite{Sanjhopf,SeFlo}. We would
like to study in detail this phenomenon for arbitrary dimension of $W^{u}(0)$
because, when it takes place, an interesting invariant object is created
near the origin, namely an attractor with the Borsuk homotopy type (or shape) of a sphere of dimension
one unit less than the dimension of $W^{u}(0).$

In order to state our next result, we must introduce first a definition
which is applicable in the following situation: Let $\varphi _{\lambda }:%
\mathbb{R}^{n}\times \mathbb{R}\rightarrow \mathbb{R}^{n}$ be a family of
flows induced by a system $\dot{X}=F_{\lambda }(X)$ of ODE in $\mathbb{R}^{n}
$ which depend differentiably on a parameter $\lambda \in \lbrack 0,1]$ and
suppose that $0$ is an equilibrium for every $\lambda $. Suppose,
additionaly, that $0$ is an attractor for $\lambda =0$ and a hyperbolic
fixed point with exactly $k$ positive and $n-k$ negative eigenvalues for $%
\lambda >0$ (hence, $W_{\lambda }^{s}(0)$ is an immersed $(n-k)-$dimensional
manifold). We say that the family is \textit{rigid }at $\lambda =0$ if there
is an $\varepsilon >0$ arbitrarilly small and a $\lambda _{\varepsilon}>0$ such that
every trajectory of $W_{\lambda }^{s}(0)$ other than $0$ leaves $\overline{B_\varepsilon(0)}$ in the past and the pair $(\overline{B_\varepsilon(0)},\overline{B_\varepsilon(0)}\cap W_{\lambda
}^{s}(0))$ is homeomorphic to the pair $(\overline{B_{n}},\overline{B_{n-k}})$ (the unit closed
balls of dimensions $n$ and $(n-k)$ respectively) for every $\lambda $ with $%
0<\lambda <\lambda_{\varepsilon}$. Rigidity is a kind of uniformity condition for the
local stable manifolds (which are known to be topological $(n-k)-$balls), which
prevents them from collapsing immediately after $\lambda =0$ (it is not
difficult to describe situations where that phenomenon occurs).

In the case of the Lorenz system immediately after the pitchfork bifurcation
($r>1$), the stable manifold $W^{s}(0)$ of the origin can be regarded, at
least near the origin, as a plane, the plane associated with the two
negative eigenvalues of the flow linearized near the origin (see \cite[p. 13]{Spa}).
Thus, for $\varepsilon $ sufficiently small, every trajectory other than $0$
leaves $\overline{B_{\varepsilon }(0)}$ in the past and the pair $(\overline{%
B_{\varepsilon }(0)},\overline{B_{\varepsilon }(0)}\cap W^{s}(0))$ is homeomorphic
to the pair $(\overline{B_{3}},\overline{B_{2}})$ for all values of the
parameter sufficiently close to $r=1$. Hence the Lorenz system is rigid at
the value of the  pitchfork bifurcation. 

\begin{theorem}\label{pitchfork}
 Let $\varphi _{\lambda }:\mathbb{R}^{n}\times \mathbb{R}\rightarrow 
\mathbb{R}^{n}$ be a family of flows induced by a system $\dot{X}=F_{\lambda
}(X)$ of ODE in $\mathbb{R}^{n}$ depending differentiably on a parameter $%
\lambda \in \lbrack 0,1]$ and suppose that $0$ is an equilibrium for every $%
\lambda $. Suppose that $\{0\}$ is an attractor for $\lambda =0$ and a
hyperbolic fixed point with exactly $k$ positive and $n-k$ negative eigenvalues for $\lambda >0$%
. We assume that $W_{\lambda }^{s}(0)$ is rigid at $\lambda =0$. Then there
exists a $\lambda _{0}>0$ such that for every $\lambda $ with $0<\lambda
<\lambda _{0}$ there exists an attractor $A_{\lambda }$ with the Borsuk
homotopy type (shape) and the cohomology of the sphere $S^{k-1}$. The Conley index of $%
A_{\lambda }$ is the homotopy type of $(S^{k-1}\cup \{\ast \},\ast)$  and its
cohomological Conley index is $\mathbb{Z}$ for $i=0,k-1$ and $\{0\}$ 
otherwise if $k>1$ and $\mathbb{Z}\oplus\mathbb{Z}$ for $i=0$ and $\{0\}$ otherwise if $k=1$. Moreover, the family $A_{\lambda }$ shrinks to $0$ when $\lambda
\rightarrow 0$ (in particular, if $k=2$ we have a family of attractors with
the shape of $S^{1}$ shrinking to $0$). Moreover, the attractor $A_{\lambda
} $ is contained in an attractor $K_{\lambda }$ of trivial shape which
contains the origin and such that $(A_{\lambda},\{0\})$ is an
attractor-repeller decomposition of $K_{\lambda }$ whose Morse equations are
\[
1+t^{k-1}+t^k=1+(1+t)t^{k-1}.
\]
The family $(K_{\lambda })$ also shrinks to $0$. In the particular case
of the Lorenz flow, $A_{\lambda }$ consists of two equilibria, i.e. $%
A_{\lambda }=S^{0}$ and $K_\lambda$ is the union of $A_\lambda$ with the unstable manifold of the origin i.e. $K_\lambda\approx \overline{B_1}$ and the Morse equations are 

\[
2+t=1+(1+t).
\]
\end{theorem}

\begin{proof}
Since $\{0\}$ is an attractor of $%
\varphi _{\lambda }$ for $\lambda =0$, there is a continuation $(K_{\lambda
})$ of $\{0\}$, where $K_{\lambda }$ is an attractor of trivial shape of $\varphi _{\lambda }$ 
for $\lambda$ sufficiently small \cite[Theorem~4]{sanjuni}. Since $K_{\lambda }$ is the maximal invariant set of $%
\varphi _{\lambda }$ in a neighborhood of $0$ and $0$ is an equilibrium of $%
\varphi _{\lambda }$ we have that $0\in K_{\lambda }$. We define $A_{\lambda
}=K_{\lambda }\setminus W_{\lambda }^{u}(0)$. First we see that $W_{\lambda }^{u}(0)\subset
K_{\lambda }.$ Otherwise, there exists a point $x\in W_{\lambda }^{u}(0)$
such that $x\notin K_{\lambda }$ and we arrive at a contradiction as follows$%
.$ Notice that $x$ must be in the region of attraction $\mathcal{A}%
_{\lambda }(K_\lambda)$ of $K_{\lambda }$ since $0\in K_{\lambda }$ and $\mathcal{A}%
_{\lambda }(K_\lambda)$ is a neighborhood of $K_{\lambda }$, which implies that $xt$
must be in $\mathcal{A}_{\lambda }(K_\lambda)\setminus K_\lambda$ for certain negative value ot $t$. Since 
$\mathcal{A}_{\lambda }(K_\lambda)\setminus K_\lambda$ is invariant, the whole trajectory of  $x$ 
is in $\mathcal{A}_{\lambda }(K_\lambda)\setminus K_\lambda$ and, hence, $\emptyset\neq\omega_\lambda (x)\subset K_{\lambda }$. In addition, since $x\in W_\lambda^u(0)$,  $\omega^*(x)=\{0\}\subset K_\lambda$. As a consecuence, $K_\lambda\cup\varphi_\lambda(x,\mathbb{R})$ is a compact invariant set contained in $\mathcal{A}_\lambda(K_\lambda)$. This contradicts the stability of $K_\lambda$, since $K_\lambda$ being stable, must be the maximal compact invariant set contained in its region of attraction. We shall see now that $A_{\lambda }$ is compact. Since $%
(K_{\lambda })$ is a continuation of $\{0\}$ we have that $K_{\lambda }\subset
B_\varepsilon(0)$ for $0<\lambda<\lambda_0$ where $\varepsilon>0$ is chosen using the
ridigity of $W_{\lambda }^{s}(0)$. If $A_{\lambda }$ is not compact then
there exists a sequence of points $x_{n}\in A_{\lambda }=K_{\lambda
}\setminus W_{\lambda }^{u}(0)$ such that $x_{n}\rightarrow x$ and $x\in W_{\lambda
}^{u}(0).$ Since $K_{\lambda }$ and $W_{\lambda }^{u}(0)$ are invariant then 
$A_{\lambda }$ is also invariant and we can assume that $x_{n}\rightarrow 0$%
. This is proved as follows. Since $x\in W_{\lambda }^{u}(0)$ we have that $%
xt_{k}\rightarrow 0$ for a certain sequence $t_{k}\rightarrow -\infty $. For
every $k$ select a $x_{n_{k}}$ such that $x_{n_{k}}t_{k}$ \ is $1/k$-close
to $xt_{k}$.  Hence $x_{n_{k}}t_{k}\rightarrow 0$ with $x_{n_{k}}t_{k}\in $ 
$A_{\lambda }$. Consider now a Hartman-Grobman block $%
H_{\lambda }$ for $\varphi _{\lambda }$ contained in $B_\varepsilon(0)$. Since
the points $x_{n}$ are not in $W_{\lambda }^{u}(0)$ there exists, for each $n$, $t_n<0$ such that $x_nt_n\in\partial H_\lambda$ and $x_n[t_n,0]\subset H_\lambda$. Since $\partial H_\lambda$ is compact we may assume that $x_nt_n\to y\in\partial H_\lambda$. Notice that the sequence $t_n\to-\infty$ since, otherwise, we may assume that it converges to some $t_0\leq 0$ and, hence $x_nt_n$ converges to  $0t_0=0$  which is clearly not in $\partial H_\lambda$. Let us see that $yt\in H_\lambda$ for each $t\geq 0$ and, as a consequence, $y\in W^s_\lambda(0)$. Let $t\geq 0$, then, since $t_n\to-\infty$,  there exists $n_0$ such that $t+t_n<0$ for every $n\geq n_0$. Thus $x_n(t+t_n)\in H_\lambda$ for each $n\geq n_0$ and, since the sequence $x_n(t_n+t)$ converges to $yt$,  it follows from the compactness of $H_\lambda$ that $yt\in H_\lambda$. As a consequence, the rigidity condition ensures that the trajectory of $y$ must leave $\overline{B_\varepsilon(0)}$ and, thus, $K_\lambda$, which is in contradiction with the invariance of $K_\lambda$. This contradiction proves the compactness of $A_\lambda$.  Moreover the pair $(A_{\lambda },\{0\})$ is an attractor-repeller
 decomposition of $K_{\lambda }$. Indeed, we see that $\{0\}$ is a repeller for $\varphi_\lambda|_{K_\lambda}$. Suppose that $\{0\}$ is not a repeller for $\varphi_\lambda|_{K_\lambda}$, then \cite[Lemma~3.1]{Sal} ensures that any compact neighborhood $U$ of $0$ in $K_\lambda$ disjoint from $A_\lambda$ contains a point $x$, other than $0$, such that $\gamma^+(x)\subset U$. Since $U$ isolates $\{0\}$ in $K_\lambda$, it follows that $\omega_\lambda(x)=\{0\}$ and the rigidity condition ensures that the trajectory of $x$ must leave $K_\lambda$ in contradiction with the invariance of $K_\lambda$. Notice that $W^u_\lambda(0)$ is the region of repulsion of $\{0\}$ and, hence, $A_\lambda=K_\lambda\setminus W^u_\lambda(0)$ is its complementary attractor. Since $K_{\lambda }$ is an attractor and $A_{\lambda }$ is
an attractor in $K_{\lambda }$ then $A_{\lambda }$ is an attractor of the
flow $\varphi _{\lambda }$. We consider the attractor-repeller cohomology
sequence of the decomposition $(A_{\lambda },\{0\})$ of $K_{\lambda }$

\[
\cdots\rightarrow CH^{i-1}(K_\lambda)\rightarrow CH^{i-1}(A_\lambda)\rightarrow CH^i(\{0\})\rightarrow CH^i(K_\lambda)\rightarrow\cdots
\]

Since $K_{\lambda }$ is a continuation of the attractor $\{0\}$ of $\varphi _{0}$
we know its cohomology index and the cohomology index of $\{0\}$ for $\varphi
_{\lambda}$( because $0$ is now a hyperbolic point). We deduce from
this the cohomology index of $A_{\lambda }$ which is $\mathbb{Z}$ in
dimension $k-1$ when $k>1$ and $\mathbb{Z}\oplus\mathbb{Z}$ in dimension $0$ when $k=1$. On the other hand by the rigidity condition $\loc W_{\lambda
}^{s}(0)$ is uniformly locally flat and this implies that if we take $\delta 
$ sufficiently small we have $\overline {B_\delta(0)}$ is contained in the region of
attraction of $K_{\lambda }$ and $\overline{ B_\delta(0)}\setminus W_{\lambda
}^{s}(0)$ is homeomorphic to $\overline {B_n}\setminus \overline{B_{n-k}}$ which is homotopy equivalent to $%
S^{k-1}$. Using  the flow we can define a sequence of maps $r_k:\overline{B_\delta(0)}\setminus W_{\lambda }^{s}(0)\to\mathbb{R}^n$ by
\[
r_k(x)=\varphi_\lambda(x,k).
\]
Since $A_\lambda$ is an attractor and $\overline{B_\delta(0)}\setminus W_\lambda^s(0)$ is contained in its region of attraction, it follows that given any neighborhood $U$ of $A_\lambda$ there exists $k_0\in\mathbb{N}$ such that the image of $r_k$ is contained in $U$ for every $k\geq k_0$. In addition, the flow defines, in a natural way, a homotopy between $r_k$ and $r_{k+1}$, for each $k\geq k_0$, taking place in $U$. As a consequence, this family of maps defines an \emph{approximative sequence}
\[
\mathbf{r}=\{r_{k},\overline {B_\delta(0)}\setminus W_{\lambda }^{s}(0)\rightarrow A_{\lambda }\},
\]
in the sense of Borsuk \cite{Bormono} and, hence, a shape morphism. Since $r_k|_{A_\lambda}$ is homotopic to the identity  for each $k$, it follows that the shape morphism induced by the inclusion $i:A_\lambda\hookrightarrow\overline {B_\delta(0)}\setminus W_{\lambda }^{s}(0)$ is a left inverse for $\mathbf{r}$ and, therefore 
\[
\Sh(S^{k-1})=\Sh(\overline {B_\delta(0)}\setminus W_{\lambda }^{s}(0))\geq \Sh(A_{\lambda }). 
\]%
On the other hand, since the cohomology Conley index of $A_{\lambda }$
is $\mathbb{Z}$ in dimension $k-1$ if $k>1$ and $\mathbb{Z}\oplus\mathbb{Z}$ in dimension $0$ if $k=1$, it follows that $H^*(A_{\lambda })\neq H^*(\{*\}).$  Now since $\Sh(S^{k-1})\geq \Sh(A_{\lambda
})$ and $H^{*}(A_{\lambda })\neq H^*(\{*\})$ Borsuk-Holszty\'nski Theorem \cite{Bor-Hol}, which ensures that if a compactum $K$ satisfies that $\Sh(K)\leq \Sh(S^n)$ and  $K$ does not have the shape of a point then $\Sh(K)=\Sh(S^n)$, applies and we
have that, in fact, $\Sh(S^{k-1})=\Sh(A_{\lambda })$. A direct consequence of this fact is that the Conley index of $A_\lambda$ is the homotopy type of $(S^{k-1}\cup \{*\},*)$.

From the previous discussion it readily follows that $CH^i(A_\lambda)$ is $\mathbb{Z}$ for $i=0,k-1$ and zero otherwise when $k>1$ and $\mathbb{Z}\oplus\mathbb{Z}$ for $i=0$ an zero otherwise when $k=1$, $CH^i(K_\lambda)$ is $\mathbb{Z}$ for $i=0$ and zero otherwise and $CH^i(\{0\})$ is $\mathbb{Z}$ for $i=k-1$ and zero otherwise. Combining all of this we get the desired Morse equations for the attractor-repeller decomposition $(A_\lambda,0)$ of $K_\lambda$. 
\end{proof}

\begin{remark}
Our previous result can be looked at as describing either a generalized
pitchfork bifurcation or a generalized Poincar\'{e}-Andronov-Hopf
bifurcation of arbitrary codimension.
\end{remark}

\section{Transient chaos}\label{sec:preturbulence}

For a parameter value approximately equal to $13.926...$, the behaviour of
the flow experiments an important change. At this critical value the stable
manifold of the origin includes the unstable manifold of the origin; i.e.
trajectories started in the unstable manifold of the origin tend, in both
positive and negative time, to the origin. As a consequence, a couple of
homoclinic orbits are produced, one for every branch of the unstable
manifold and we say that a \textit{homoclinic bifurcation }has taken place
at the parameter value $r_{H}=$ $13.926...$ This parameter value signals the
appearance of a phenomenon known as \textit{preturbulence} or \emph{transient chaos}, whose study was
carried out by Kaplan and Yorke and by Yorke and Yorke in \cite{K-Y, YY}. This
phenomenon is characterized by the fact that certain trajectories behave
chaotically for a while, before escaping to an external attractor. Turbulent
trajectories also exist but represent a set of measure zero. By using
arguments similar, to a certain extent, to Smale's horseshoe \cite{SmBu} they proved
that for $r>r_{H}$ \ a \ countable infinity of periodic orbits is created
together with an uncountable infinity of bounded trajectories that are
asymptotically periodic (in either forwards of backwards time) and an
uncountable infinity of bounded aperiodic trajectories. These aperiodic
trajectories were termed as \textit{turbulent} by Ruelle and Takens \cite{R-T}
because their limit sets are neither points, nor periodic orbits, nor
manifolds. Sparrow remarked that also an uncountable infinity of \textit{%
bounded trajectories which terminate in the origin }is produced. The union
of all these trajectories together with the origin forms an invariant
``strange set'' $K_{r}$ which exhibits sensitive dependence on initial
conditions. By studying a return map of the flow with respect to a suitable
Poincar\'{e} section, Sparrow proved, relying on Kaplan and Yorke's results,
that the intersections of the trajectories of $K_{r}$ with the Poincar\'{e}
section can be coded by bisequences of two symbols $S$ and $T$ such that
repeating sequences correspond to periodic orbits, sequences which terminate
on the right correspond to trajectories which terminate in the origin and
aperiodic sequences correspond to trajectories which oscilate aperiodically.

In the sequel we analyze the nature and the evolution of the strange sets $%
K_{r}$ from the point of view of Conley's index theory. To get our
conclusions, we use some facts that have been established by Kaplan-Yorke \cite{K-Y}
and Sparrow \cite{Spa}.

\subsection*{1}  \textbf{The strange sets are isolated invariant sets and they define a
continuation (in the sense of Conley's theory) of the double homoclinic loop.}%

Every $K_{r}$ is a compact isolated invariant set. As a matter of fact, if we
take a neighborhood $\mathcal{N}_{r}$ of the double homoclinic loop,
consisting of a small box $B$ around the origin, together with two
tubes, $S$ and $T$ around the two branches of the loop (see \cite[Appendix D, pg. 199]{Spa}), we have that $K_{r}$
is the maximal invariant set inside this neighborhood for values of $r$
close to that of the homoclinic bifurcation. The passage of the trajectories
of $K_{r}$ through the tubes is in correspondence with the codification with
two symbols previously stated, and this is the explanation for the use of
the same notation. We clearly have that the family $(K_{r})$, for $r>r_{H}$,
is a continuation (in the sense of Conley's theory), of the double
homoclinic loop which originates the homoclinic bifurcation at $r=r_{H}$ .

\subsection*{2} \textbf{The continuation is continuous in the Hausdorff metric for $
r=r_{H}$}.

As a matter of fact, each tube $S$ and $T$ contains exactly one periodic
orbit which does not wind around the $z-$axis. The notation $S$ and $T$ is
also used to designate these two simplest orbits. Then, if we fix $\varepsilon
>0$ we have that the neighborhood $\mathcal{N}_{r}$ can be chosen to be 
contained in the $\varepsilon -$neighborhood of the double homoclinic loop for
values of $r$ sufficiently close to $r_{H}$ and, hence, so is $K_r$. On the other hand, the $%
\varepsilon -$neighborhood of the orbits $S$ and $T$ (and hence the $\varepsilon -$%
neighborhood of $K_{r}$) contains the double homoclinic loop for $r$
sufficiently close to $r_{H}$. This proves that $K_{r}$ converges to the
double homoclinic loop when $r\rightarrow r_{H}$.

\subsection*{3} \textbf{The strange sets have the cohomological Conley index of the circle.}

The cohomological Conley index of $K_{r}$ is isomorphic to $H^{\ast
}(S^{1},\ast )$, where $S^{1}$ is the circle. This is a consequence of the
fact that the origin $\{0\}$ is a continuation of the double homoclinic loop
for $r<r_{H}$ . Since the cohomological Conley index is preserved by
continuation and the index of the origin is isomorphic to $H^{\ast
}(S^{1},\ast )$ then the index of the double homoclinic loop and also that
of its continuation $K_{r}$ for $r>r_{H}$ must be the same.

\subsection*{4}  \textbf{The strange sets are repellers in an attractor-repeller
decomposition of the global attractor  $\Omega _{r}$ of the flow.}

The strange set $K_{r}$ is contained in the global attractor $\Omega _{r}$.
Since all the trajectories in $\Omega _{r}$ not contained in $K_{r}$
terminate in $C_{1}$ or $C_{2}$ we must have that the $\omega ^{\ast }-$%
limit of these trajectories (else than $C_{1}$ or $C_{2}$) must be contained
in $K_{r}$. As a consequence $(\{C_{1,}C_{2}\},K_{r})$ is an
attractor-repeller decomposition of the global attractor $\Omega _{r}$.

\subsection*{5} \textbf{The strange sets $K_{r}$ are not chaotic but they admit
an attractor-repeller decomposition $(\{0\},L_{r})$  where $L_{r}$ 
is chaotic. The set $L_{r}$ is the suspension of a Smale
horseshoe but the strange set $K_{r}$ is not.}

Contrarily to some statements in the literature, the strange set $K_{r}$ is
not chaotic, since there is not a single trajectory in $K_{r}$ whose closure
contains the trajectories terminating in the origin. This was remarked by
Sparrow in \cite{Spa}. However, if we consider all the trajectories in $K_{r}$
except those terminanting in the origin we obtain a chaotic invariant set $%
L_{r}$. As a matter of fact, this was the set discovered and studied by
Kaplan and Yorke in \cite{K-Y} where they proved that $L_{r}$ has sensitive
dependence on initial conditions, the set of periodic orbits is dense in $%
L_{r}$ and it contains an uncountable infinity of aperiodic dense
trajectories. This set is the suspension of a return map of the flow with
respect to a suitable Poincar\'{e} section studied by Sparrow, whose
dynamics is that of the Smale horseshoe. On the other hand the existence of
a fixed point in $K_{r}$ prevents the strange set from being a suspension.
As we prove in our next result, $L_{r}$ is an isolated invariant set with
trivial cohomological index and the pair $(\{0\},L_{r})$ defines an
attractor-repeller decomposition of $K_{r}$. We deduce from this that 
$\{\{C_{1,}C_{2}\},\{0\},L_{r}\}$ is a Morse decomposition of the
global attractor $\Omega _{r}$. The Morse equations of this decomposition
are obtained also in our next result, where we analyze a situation which is
more general than the one described here.

\subsection*{6} \textbf{The strange sets $K_{r}$ have the cohomology of the
figure eight.}

In spite of its dynamical and topological complexity, the strange set $K_{r}$
has the cohomology of the figure eight. This is a consequence of a more
general result proved in our next theorem.

\bigskip

Our study of the evolution of the strange set concerns mainly asymptotic
properties of its internal structure and of the structure of the global
attractor of the flow. Recently, E.J. Doedel, B. Krauskopf and H.M. Osinga
\cite{DKO} have performed a study of the global organization
of the phase space in the transition to chaos where they show how global
invariant manifolds of equilibria and periodic orbits change with the
parameters.

\bigskip

The following is a result of a general nature which has been suggested by
the previous discussion on the evolution of the Lorenz strange set. Some of
the remarks previously made are consequences of this theorem.

\begin{theorem}
Let $\varphi _{\lambda }:\mathbb{R}^{3}\times \mathbb{R}\rightarrow \mathbb{R%
}^{3}$ be a dissipative family of flows induced by a system $\dot{X}%
=F_{\lambda }(X)$ of ODE in $\mathbb{R}^{3}$ depending differentiably on a
parameter $\lambda \in \lbrack 0,1]$. Suppose that $0$ is a hyperbolic
equilibrium for every $\lambda $ with exactly one positive and two negative
eigenvalues and that there are two other hyperbolic equilibria $C_{1}$ and $%
C_{2}$, both of them having one real negative eigenvalue $\beta _{\lambda }$
and two conjugate complex eigenvalues $\mu _{\lambda }\pm \upsilon _{\lambda
}i$ with $\mu _{\lambda }<0$ for every $\lambda $. Suppose that for $\lambda
=0$ the fixed point $0$ has two homoclinic trajectories corresponding with
the two branches of its unstable manifold and that the points $C_{1}$ and $%
C_{2}$ attract all bounded orbits of $\mathbb{R}^{3}$ not lying in $%
W_{0}^{s}(0)$ and suppose, additionally, that for $\lambda >0$ the two
branches of $W_{\lambda }^{u}(0)$ connect the point $0$ with $C_{1}$ and $%
C_{2}$ respectively and that $W_{\lambda }^{s}(0)\setminus\{0\}$ contains at least one 
bounded orbit. Then:

\begin{enumerate}

\item[a)] For $\lambda =0$, the $\omega ^{\ast }$-limit of every bounded orbit
different from the stationary orbits $C_{1}$ and $C_{2}$ is contained in the
double homoclinic loop $W_{0}^{u}(0)$.

\item [b)] For $\lambda >0$ the set of bounded trajectories of $\varphi _{\lambda }$
other than those finishing in $C_{1}$ or $C_{2}$ is a non-empty isolated invariant set 
$K_{\lambda }$ whose cohomology Conley index is isomorphic to $H^{\ast
}(S^{1},\ast )$. Moreover $(\{C_{1,}C_{2}\},K_{\lambda })$ is an
attractor-repeller decomposition of the global attractor $\Omega _{\lambda }$
and $K_{\lambda }$ itself has a finer attractor-repeller decomposition $%
(\{0\},L_{\lambda })$ where $L_{\lambda }$ consists of all bounded
trajectories not ending neither in the origin nor in $C_{1}$ or $C_{2}$. The
set $L_{\lambda }$ has trivial cohomology index and the triple $%
\{\{C_{1,}C_{2}\},\{0\},L_{\lambda }\}$ is a Morse decomposition of the global
attractor whose Morse equations are
\[
2+t=1+(1+t).
\]
\item[c)] If the complex invariant manifolds of the points $C_1$ and $C_2$ consist of all the bounded orbits
finishing in $C_{1}$ and $C_{2}$ respectively (as is the case in the Lorenz equations)
then the cohomology of $K_{\lambda }$ agrees with that of the figure eight.
\end{enumerate}
\end{theorem}

\begin{proof}
To prove part a) consider the global attractor $\Omega _{0}$ of the flow $%
\varphi _{0}$. Since $\{C_{1,}C_{2}\}$ is an attractor contained in $\Omega
_{0}$ there exists a dual repeller for the flow $\varphi _{0}$ restricted to 
$\Omega _{0}$. Obviously the double loop $W_{0}^{u}(0)$ is contained in this
repeller. Moreover, for every point $x\in \Omega _{0}$ with $x\neq C_i$, $i=1,2$, we have that $\emptyset\neq\omega ^{\ast }(x)\subset W_{0}^{u}(0)$ since,
otherwise, there would be a bounded orbit in $\omega ^{\ast }(x)$ not lying
in $W_{0}^{s}(0)$ \ and not attracted by $C_{1}$ or $C_{2}$, contrarily to
our hypothesis. As a consequence, $W_{0}^{u}(0)$ is, in fact, the dual
repeller of $\{C_{1,}C_{2}\}$ for the flow $\varphi _{0}$ restricted to $%
\Omega _{0}$.

To prove part b) we use the fact that the attractor-repeller decomposition $%
(\{C_{1,}C_{2}\},W_{0}^{u}(0))$ of $\Omega _{0}$ has a continuation to an
attractor-repeller decomposition of the global attractor $\Omega _{\lambda }$
of the flow $\varphi _{\lambda }$. The continuation of $\{C_{1,}C_{2}\}$ is
the attractor $\{C_{1,}C_{2}\}$ itself. And the continuation of the repeller 
$W_{0}^{u}(0)$ is the set $K_{\lambda }$ formed by the union of all bounded
orbits not ending in $C_{1,}$or $C_{2}$, which is the dual repeller of $%
\{C_{1,}C_{2}\}$ for the restriction of $\varphi _{\lambda }$ to $\Omega
_{\lambda }$. Since the Conley index continues, the cohomology index of $%
K_{\lambda }$ for the flow $\varphi _{\lambda }$ must agree with that of $%
W_{0}^{u}(0)$ for the flow $\varphi _{0}$. We see that the cohomology index of $W_{0}^{u}(0)$ is isomorphic to 
$H^{\ast }(S^{1},\ast )$. Let $N$ be a compact manifold with boundary which is a positively invariant neighborhood of the global attractor $\Omega_0$. It is possible to get such a neighborhood by using a Lyapunov function. Then, $(N,\emptyset)$ is an index pair for $\Omega_0$. Notice that by \cite[Theorem~3.6]{KapRod} $N$ is acyclic. Since $(\{C_1, C_2\},W^u(0))$ is an attractor-repeller decomposition for $\Omega_0$  \cite[Corollary~4.4]{Sal} ensures the existence of  $N_0\subset N$ such that  $(N_0,\emptyset)$ is an index pair for  $\{C_1, C_2\}$ and $(N,N_0)$ is an index pair for $W^u(0)$. Taking into account that $N_0$ is a positively invariant neighborhood of $\{C_1, C_2\}$, \cite[Theorem~3.6]{KapRod} it follows that $H^k(N_0)$ is isomorphic to $\mathbb{Z}\oplus \mathbb{Z}$ for $k=0$ and zero otherwise. Consider the long exact sequence of reduced cohomology of the pair $(N,N_0)$
\[
\cdots \to\widetilde{H}^k(N,N_0)\to \widetilde{H}^k(N)\to \widetilde{H}^k(N_0)\to \cdots
\]
This exact sequence, together with the previous discussion, ensure that $H^*(N,N_0)$ is isomorphic to $0$ if $i\neq 1$.  On the other hand, for $i=1$ we have
\[
0\cong\widetilde{H}^0(N)\to \mathbb{Z}\cong\widetilde{H}^0(N_0)\xrightarrow{\partial}\widetilde{H}^1(N,N_0)\to \widetilde{H}^1(N)\cong\{0\}
\]
hence, $\partial$ is an isomorphism and, as a consequence, $H^1(N,N_0)\cong\mathbb{Z}$.

Now consider the subset $L_{\lambda }$ of $K_{\lambda }$ consisting of all
bounded trajectories of $\varphi _{\lambda }$ not ending neither in the
origin nor in $C_{1}$ or $C_{2}$. We shall prove that $L_{\lambda }$ is a
repeller for the flow $\varphi _{\lambda }$ restricted to $K_{\lambda }$. We
remark that $W_{\lambda }^{u}(0)\cap K_{\lambda }=\{0\}$ for $\lambda >0$
since the two branches of $W_{\lambda }^{u}(0)$ connect the point $0$ with $%
C_{1}$ and $C_{2}$ respectively and the stationary points $C_{1}$ and $C_{2}$
do not belong to $K_{\lambda }$. Since $0$ is a hyperbolic equilibrium for
every $\lambda $ with exactly one positive and two negative eigenvalues,it possesses a 
Hartman-Grobman block $%
H_\lambda$ of $0$ (which can be arbitrarily small)$.$ We claim that there exists an $%
\varepsilon >0$ such that for every $x\in H_\lambda\cap $ $K_{\lambda }$ with $x\in
B_\varepsilon(0)$ its positive semitrajectory $\gamma^{+}(x)$ is contained in 
$H_\lambda$ and, hence, ends in $0$. Otherwise there is a sequence of points $%
x_{n}\in K_{\lambda }$, $x_{n}\rightarrow 0$, such that $\gamma ^{+}(x_{n})$
leaves $H_\lambda$. This produces an orbit in $K_{\lambda }$ which leaves $H_\lambda$ in the
future and whose $\omega ^{\ast }$-limit is $\{0\}$, which is in contradiction
with the fact that $W_{\lambda }^{u}(0)\cap K_{\lambda }=\{0\}$. Then there
is an $\varepsilon >0$ such that all points of $H_\lambda\cap K_{\lambda }$ contained
in the ball $B_\varepsilon(0)$ go to $0$. Hence, we have a neighborhood $H_\lambda\cap
K_{\lambda }\cap B_\varepsilon(0)$ of $0$ in $K_{\lambda }$ attracted by $%
\{0\} $ and such that the orbits of its points do not leave $H_\lambda$ in the
future. This proves that $\{0\}$ is an attractor in $K_{\lambda }$ whose dual
repeller is obviously $L_{\lambda }$.  The
attractor-repeller cohomology exact sequence of the decomposition $(\{0\},L_{\lambda })$
of $K_{\lambda }$ takes de form
\[
\cdots\xrightarrow{\partial} CH^*(L_\lambda)\xrightarrow{j^*} CH^*(K_\lambda)\xrightarrow{i^*} CH^*(\{0\})\xrightarrow{\partial}\cdots
\]
 and, taking into account that the cohomology indices of $%
\{0\} $ and $K_{\lambda }$ are both $H^{\ast }(S^{1},\ast )$, we readily get
that $CH^i(L_{\lambda })$ is trivial for $i\neq 1,2$. To see that $CH^i(L_\lambda)$ is trivial for $i=1,2$ we analyse the following segment of the long exact sequence
\[
0\to CH^1(L_\lambda)\xrightarrow{j^*} CH^1(K_\lambda)\xrightarrow{i^*} CH^1(\{0\})\xrightarrow{\partial} CH^2(L_\lambda)\to 0
\]

Let us see that $i^*$ is an isomorphism. Let $\bar{N}$ be a compact manifold with boundary which is a positively invariant neighborhood of the global attractor $\Omega_\lambda$. Then, $(\bar{N},\emptyset)$ is an index pair for $\Omega_\lambda$. Since $(\{C_1, C_2\},K_\lambda)$ is an attractor-repeller decomposition for $\Omega_\lambda$ and $(\{0\},L_\lambda)$ is an attractor-repeller decomposition of $K_\lambda$, it easily follows that $\{\{C_1, C_2\},\{0\},L_\lambda\}$ is a Morse decomposition of $\Omega_\lambda$. Hence, \cite[Corollary~4.4]{Sal} ensures the existence of a filtration $\bar{N}_0\subset \bar{N}_1\subset\bar{N}$ such that  $(\bar{N}_0,\emptyset)$ is an index pair for  $\{C_1, C_2\}$, $(\bar{N},\bar{N}_0)$ is an index pair for $K_\lambda$,  $(\bar{N},\bar{N}_1)$ is an index pair for $L_\lambda$ and $(\bar{N}_1,\bar{N}_0)$ is an index pair for $\{0\}$. Notice that the homomorphism $i^*$ is induced by the inclusion $i:(\bar{N}_1,\bar{N}_0)\hookrightarrow(\bar{N},\bar{N}_0)$. Taking into account that $\bar{N}_0$ is a positively invariant neighborhood of $\{C_1, C_2\}$, \cite[Theorem~3.6]{KapRod} ensures that this inclusion induces the following conmutative diagram of short exact sequences in cohomology
\[
\begin{CD}
0@>>>\mathbb{Z}\cong \widetilde{H}^0(\bar{N}_0)@>\partial>> \mathbb{Z}\cong H^1(\bar{N},\bar{N}_0) @>>>H^1(\bar{N})@>>>0\\
&& @VVV  @Vi^*VV @VVV \\
0@>>> \mathbb{Z}\cong \widetilde{H}^0(\bar{N}_0)@>\bar{\partial}>> \mathbb{Z}\cong H^1(\bar{N}_1,\bar{N}_0) @>>>H^1(\bar{N}_1)@>>>0
\end{CD}
\]
Since $H^1(\bar{N})=0$ by \cite[Theorem~3.6]{KapRod}, it follows that $\partial$ is an isomorphism. Let us see that $\bar{\partial}$ is also an isomorphism. From the fact that the lower right arrow is an epimorphism, it follows that $H^1(\bar{N}_1)$ is either $0$, $\mathbb{Z}$ or a finite cyclic group. The exactness of the second row ensures that $\bar{\partial}$ is a monomorphism and, hence, the lower right arrow cannot be an isomorphism.  As a consequence $H^1(\bar{N}_1)$ cannot be $\mathbb{Z}$. In addition, the Universal Coefficient Theorem ensures that $H^1(\bar{N}_1)$ must be torsion free and, as a consequence, it cannot be finite cyclic either. Hence, $H^1(\bar{N}_1)=0$ and $\bar{\partial}$ is also an isomorphism. By combining this with the fact that the leftmost vertical arrow is the identity homomorphism, it follows that $i^*$ is  an isomorphism and, hence, it readily follows, from the exactness of the attractor-repeller sequence, that $CH^i(L_\lambda)=0$ for $i=1,2$.

We see that the Morse equations of the decomposition $\{\{C_{1},C_{2}\},\{0\},L_{\lambda
}\}$ of the global attractor $\Omega _{\lambda }$ are
\[
2+t=1+(1+t).
\] 
  Since $\{C_1, C_2\}$ is an attractor consisting of two fixed points and $N_1$ is a positively invariant neighborhood, it easily follows that $CH^*(\{C_1,C_2\})\cong H^*(S^0)$ which contributes  with the term $2$ of the lefthand side of the equation. The term $t$ of the lefthand side comes from the fact that $CH^*(\{0\})\cong H^*(S^1,*)$ is a hyperbolic fixed point with one real positve eigenvalue and two complex conjugate eigenvalues with negative real part.  $L_\lambda$ does not contribute to the equations since its cohomology index is trivial. Finally, the term $1$ from the righthand side of the equation comes from the fact that $CH^*(\Omega_\lambda)\cong H^*(\bar{N})$ which is acyclic by \cite[Theorem~3.6]{KapRod}.   

To prove part c) we remark that our hypothesis ensures the existence of
arbitrarily small positively invariant neighborhoods $\widetilde{N}_{1}$ and $\widetilde{N}_{2}$ of $%
C_{1}$ and $C_{2}$ in $\Omega _{\lambda }$ that are topological closed disks. If we
call $\widetilde{N}=\widetilde{N}_1\cup \widetilde{N}_2$ and consider the exact cohomology sequence of the pair $(\Omega _{\lambda
},\widetilde{N})$ we readily see that $H^{k}(\Omega _{\lambda
},\widetilde{N})=\{0\}$ for every $k\neq 1$ and $H^{1}(\Omega _{\lambda },\widetilde{N})=\mathbb{Z}$. Now consider smaller positively invariant closed disks $\hat{N}_{1}$ and $\hat{N}_{2}$ contained in the interiors of $\widetilde{N}_{1}$ and $\\widetilde{N}_{2}$ respectively. By excision $H^{k}(\Omega
_{\lambda },\widetilde{N})\cong H^{k}(\Omega _{\lambda }\setminus\hat{N},\widetilde{N}\setminus\hat{N})$,
where $\hat{N}=\hat{N}_{1}\cup \hat{N}_{2}.$ By the choice of the disks $%
\hat{N}_{1}$ and $\hat{N}_{2}$ we have that $\Omega _{\lambda }\setminus\hat{N}$ is
negatively invariant and, since $K_{\lambda }$ is the complementary repeller of $\{C_1, C_2\}$ in $\Omega
_{\lambda }$, the  cohomology of $K_{\lambda }$ agrees with that of $\Omega
_{\lambda }\setminus\hat{N}$ (see \cite[Theorem~3.6]{KapRod}). By combining this with the fact that  $\Omega_\lambda\setminus \hat{N}$ is connected, since otherwise $\hat{N}_i$ would disconnect $\widetilde{N}_i$ for $i=1,2$, it follows that $K_\lambda$ is connected. If we consider now the exact cohomology
sequence of the pair $(\Omega _{\lambda }\setminus\hat{N},\widetilde{N}\setminus\hat{N})$ 
\[
\cdots\rightarrow H^{k-1}(\widetilde{N}\setminus\hat{N})\rightarrow H^{k}(\Omega _{\lambda }\setminus\hat{N},\widetilde{N}\setminus\hat{N})\rightarrow H^{k}(\Omega _{\lambda }\setminus\hat{N}%
)\rightarrow H^{k}(\widetilde{N}\setminus\hat{N})
\rightarrow \cdots
\]%
and take into account that $\widetilde{N}\setminus\hat{N}$ is homotopy equivalent to the union
of two disjoint circles we readily get that the homology of $K_{\lambda }$
is that of the figure eight.
\end{proof}

 There is some recent
literature dedicated to the study of \textit{transient chaos. }According to
Cape\'{a}ns, Sabuco, Sanju\'{a}n and Yorke \cite{CaSaSanYo} ``this is a characteristic
behaviour in nonlinear dynamics where trajectories in a certain region of
phase space behave chaotically for a while, before escaping to an external
attractor. In some situations the escapes are highly undesirable, so that
it would be necessary to avoid such a situation''. These authors have
developed control methods which prevent the escapes of the trajectories to
external attractors, in such a way that they stay in the chaotic region
forever. See \cite{DaYo, DhLa,CaSaSan,LoSan, SaSanYo, ScJe, ZaSaYo} for some contributions on this subject.

\section{Travelling repellers: the creation and evolution of the Lorenz
attractor}

The attractor-repeller decomposition $(\{0\},L_{r})$ of the strange set
ceases to exist at $r= 24.06$, when the two branches of the unstable
manifold of the origin are absorbed by $K_{r}$. As a matter of fact, they
asymptotically converge (\textit{only at this value of} $r)$ to the original
periodic orbits $S$ and $T$, responsible in the future for the Hopf
bifurcation. Immediately afterwards, the strange set $K_{r}$ expels the
simple periodic orbits $S$ and $T$ and it becomes an attractor (the Lorenz
attractor), while the unstable manifold of the origin remains in $K_{r}$. We
remark, however, that at the parameter value $r= 24.06$ the strange
set $K_{r}$ is still a repeller relative to the flow restricted to the
global attractor $\Omega _{r}$. Hence, the creation of the Lorenz attractor
is the result of a repeller-attractor bifurcation in $\Omega _{r}$ at $%
r= 24.06$.

The Conley index theory tells us that if we restrict ourselves to the
consideration of the flow $\varphi _{r}|_{\Omega _{r}}$, then the repeller $%
K_{24.06}$ \textit{continues} to a family of repellers $\hat{K}_{r}$ for
parameter values $r>24.06$. The Lorenz attractor $K_{r}$ is a proper subset
of $\hat{K}_{r}$, and it must have a dual repeller $R_{r}$. This repeller is
the union of the two original periodic orbits $S$ and $T$. We then have an
attractor-repeller decomposition $(K_{r},R_{r})$ of $\hat{K}_{r}$ for $%
r>24.06$. This discussion can be summarized as follows.

\subsection*{1} \textbf{If we consider the flow restricted to the global attractor }$%
\Omega _{r}$ \textbf{then the Lorenz attractor is created at a
repeller-attractor bifurcation of the strange set }$K_{r}$ \textbf{at the
parameter value} $r= 24.06$\textbf{:} \textbf{the strange set }$K_{r}$ 
\textbf{is a repeller for} $r= 24.06$ \textbf{and is an attractor for} 
$r>24.06$. \textbf{The continuation }$\hat{K}_{r}$ \textbf{of }$K_{24.06}$
for $r>24.06$ \textbf{is a repeller for }$\varphi _{r}|_{\Omega _{r}}$\textbf{\
which contains the Lorenz attractor} $K_{r}$ \textbf{and has an
attractor-repeller decomposition} $(K_{r},R_{r})$\textbf{, where }$R_{r}$ 
\textbf{is the union of the two original periodic orbits} $S$ \textbf{and} $%
T $.

We remark that the creation of the repeller $R_{r}$ is a necessary
consequence of the bifurcation at $r=24.06$. We can state a much more
general result, which shows that the complexity of this repeller is in some
dimensions higher than the complexity of the strange set $K_{24.06}$ (from
the point of view of Conley's theory), although its topological structure
is much simpler:

\begin{theorem}\label{attreppthm}
Let $\varphi _\lambda:\mathbb{R}^{n}\times \mathbb{R}\rightarrow \mathbb{R}^{n}$, $\lambda\in \mathbb{R}$, be a continuous family of flows and let $\Omega _\lambda$,with $\lambda_{0}\leq \lambda\leq \lambda_{1}$, be a continuation of isolated invariant sets.
Suppose that $K_{\lambda_{0}}$ is a repeller for the restricted flow $\varphi
_{\lambda_{0}}|_{\Omega _{\lambda_{0}}}$ and that there exists a family of compacta $K_\lambda$%
, with $\lambda_{0}<\lambda\leq \lambda_{1}$, such that  $K_\lambda$ is an attractor for the
restricted flow $\varphi _\lambda|_{\Omega_\lambda}$ and $K_\lambda$ converges to $%
K_{\lambda_{0}}$ in the Hausdorff metric (or, more generally, $K_\lambda$ converges
upper-semicontinuously to $K_{\lambda_{0}}$). Then a family of repellers $R_\lambda$
of $\varphi _\lambda|_{\Omega_\lambda}$, with $R_\lambda\cap K_\lambda=\emptyset$, is created
for $\lambda>\lambda_{0}$ which upper-semicontinuously converge to $K_{\lambda_{0}}$.
Moreover, if $K_{\lambda}$ has trivial cohomological Conley index in one 
dimension (as it is the case for the Lorenz attractor for dimensions other
than $0$ or $1$), then the cohomological index of $K_{\lambda_{0}}$ in that
dimension is a direct summand of that of $R_{\lambda}$. Finally, the cohomogical
indices of $K_{\lambda_{0}}$ and $R_{\lambda}$ agree in dimension $k$ if $%
K_{\lambda}$ has trivial indices in dimensions $k-1$ and $k$.
\end{theorem}

\begin{proof}
Since the family of isolated invariant compacta $\Omega _{\lambda}$ is a
continuation of $\Omega _{\lambda_{0}}$ we have that the repeller $K_{\lambda_{0}}$ of $%
\varphi _{\lambda_{0}}|_{\Omega _{\lambda_{0}}}$ continues to a family of repellers $\hat{K}%
_{\lambda}$ of $\varphi _{\lambda}|_{\Omega _{\lambda}}$. Then, for every sufficiently small
neighborhood $U$ of $K_{\lambda_{0}}$ in $\mathbb{R}^{n}$, the compactum $\hat{K}%
_{\lambda}$ is the maximal invariant set contained in $U$ for the flow $\varphi
_{\lambda}|_{\Omega _{\lambda}}$ with $\lambda$ sufficiently close to $\lambda_{0}$. Since the family
of attractors $K_{\lambda}$ converges upper-semicontinuously to $K_{\lambda_{0}}$, they
must be contained in $U$, also for $\lambda$ sufficiently small. But, since $\hat{K%
}_{\lambda}$ is maximal invariant, then $K_{\lambda}$ is, in fact, contained in $\hat{K}%
_{\lambda}$. Now, the fact that $K_{\lambda}$ is an attractor for $\varphi _{\lambda}|_{\Omega
_{\lambda}},$ and hence for $\varphi _{\lambda}|_{\hat{K}_{\lambda}}$, implies the existence of a
dual repeller $R_{\lambda}\subset \hat{K}_{\lambda}$. Since $\hat{K}_{\lambda}$ is itself a
repeller then $R_{\lambda}$ is also a repeller for the flow $\varphi _{\lambda}|_{\Omega
_{\lambda}}$ (not only for $\varphi _{\lambda}|_{\hat{K}_{\lambda}}$). Moreover, the family of
repellers $R_{\lambda}$ clearly converges upper-semicontinuously to $K_{\lambda_{0}}$
(since the family $\hat{K}_{\lambda}$ do) and, obviously, $R_{\lambda}\cap
K_{\lambda}=\emptyset $.

We have now for $\lambda>\lambda_{0}$ an attractor-repeller decomposition $(K_{\lambda},R_{\lambda}) 
$ of the isolated invariant compactum $\hat{K}_{\lambda}$. If we write the
cohomological exact sequence of this decomposition%
\begin{equation*}
\cdots\rightarrow CH^{k-1}(K_{\lambda})\overset{\delta }{\rightarrow }%
CH^{k}(R_{\lambda})\rightarrow CH^{k}(\hat{K}_{\lambda})\rightarrow CH^{k}(K_{\lambda})\overset%
{\delta }{\rightarrow }CH^{k+1}(R_{\lambda})\rightarrow \cdots
\end{equation*}%
and take into consideration the fact that $(\hat{K}_{\lambda})$ is a continuation
of $K_{\lambda_{0}}$ (and, thus, their Conley indices agree) we see that, if $%
CH^{k}(K_{\lambda})$ vanishes then $CH^{k}(R_{\lambda})\rightarrow CH^{k}(\hat{K}_{\lambda})$
is an epimorphism and, hence, $CH^{k}(K_{\lambda_{0}})\cong CH^{k}(\hat{K%
}_{\lambda})$ \ is a direct summand of $CH^{k}(R_{\lambda})$. Moreover, if $%
CH^{k-1}(K_{\lambda})$ also vanishes then $CH^{k}(R_{\lambda})\cong CH^{k}(\hat{K}%
_{\lambda})$.
\end{proof}

We remark again that, in the case of the Lorenz equations, for $r=
24.06$ the strange invariant set $K_{24.06}$ is not yet an attractor. In
fact it is a repeller for the restricted flow $\varphi _{24.06}|_{\Omega
_{24.06}}$, which contains the two branches of the unstable manifold of the
origin and the original periodic orbits $S$ and $T$. Immediately after, the
strange invariant set expels these periodics orbits (while retaining the
unstable manifold) and becomes an attractor (the Lorenz attractor). The
periodic orbits $S$ and $T$ ``travel" through the global attractor and,
finally, are absorbed by the fixed points $C_{1}$ and $C_{2}$ at the
parameter value $r= 24.74$, when a Hopf bifurcation takes place.

From the point of view of the global attractor $\Omega _{r}$, we have that
the pair $(K_{24.06},\{C_{1},C_{2}\})$ defines a repeller-attractor
decomposition of $\Omega $ while the pair $(K_{24.74},\{C_{1},C_{2}\})$
defines an attractor-repeller decomposition. The mechanism which makes
possible this sharp transformation is the expulsion by $K_{24.06}$ of the
original periodic orbits $S$ and $T$ and its posterior absortion by $C_{1}$
and $C_{2}$ at the parameter value $r=24.74$. In other words, the
``travelling repeller" $R_{r}=S\cup T$ is responsible for the transition. We
summarize the process in the following statement.

\subsection*{2} \textit{(From repeller-attractor to attractor-repeller decompositions of }%
$\Omega _{r}$). \textbf{The strange set }$K_{24.06}$ \textbf{is a repeller
relative to the restricted flow }$\varphi _{24.06}|_{\Omega _{24.06}}$, \textbf{%
which contains the two branches of the unstable manifold of the origin and
the original periodic orbits} $S$ \textbf{and} $T$. \textbf{The pair }$%
(K_{24.06},\{C_{1},C_{2}\})$ \textbf{defines a repeller-attractor
decomposition of} \textbf{the global attractor} $\Omega _{24.06}$. \textbf{%
Immediately after (i.e. for }$r>24.06$)\textbf{, the strange invariant set
expels these periodics orbits (while retaining the unstable manifold) and
becomes an attractor (the Lorenz attractor). The set }$R_{r}=S\cup T$ 
\textbf{is a repeller relative to the flow }$\varphi _{r}|_{\Omega _{r}}$ 
\textbf{and ``travels" through } $\Omega _{r}$ \textbf{until finally is absorbed
by }$\{C_{1},C_{2}\}$ \textbf{at the parameter value }$r=24.74$ 
\textbf{of the Hopf bifurcation. At this value, the pair }$%
(K_{24.74},\{C_{1},C_{2}\})$ \textbf{defines an attractor-repeller } \textbf{%
decomposition of }$\Omega _{24.74}.$

Now a few comments about the topological properties of the Lorenz attractor
are in order. Some global properties of the Lorenz attractor have been
studied in \cite{Sanjhopf}. In particular, the Borsuk homotopy type (or shape) of the
attractor is calculated there and from this calculation all the homological
and cohomological invariants follow. Another possibility for studying the
global properties of the attractor is to use the branched manifold (see figure~\ref{fig:1}). We give only a brief, informal, indication on how this
can be done.

\begin{figure}[h]
\includegraphics[scale=.4]{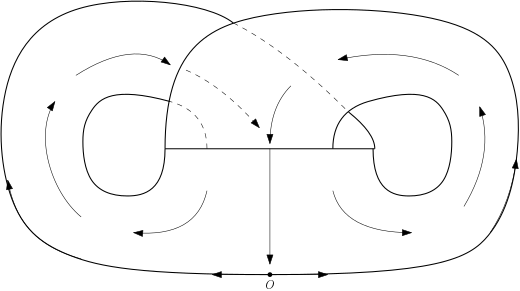}
\caption{Branched manifold}
\label{fig:1}
\end{figure}

The branched manifold is a two-dimensional manifold with singularities (the
branch points) on which the forward flow (i.e. a semi-flow) is defined. In
spite of its name, it is not a manifold but an Absolute Neighborhood Retract
(ANR), an important notion of the Theory of Retracts also studied by Borsuk.
The class of ANRs has homotopical properties similar to those of the
manifolds. The semi-flow in the branched manifold comes from the Lorenz flow
after collapsing to a point certain segments, all whose points share a
common future (see \cite[Appendix G, pg. 229]{Spa} for a discussion). The semi-flow has a global
attractor whose Borsuk homotopy type is the same as that of the Lorenz
attractor, since the above mentioned identification preserves the global
properties of the attractor. By \cite[Theorem~3.6]{KapRod} (see also \cite{GuSe,SanMul,Sanjsta,GMRSNLA}) the
inclusion of the attractor in the manifold is a (Borsuk) homotopy
equivalence. It follows from this that the Borsuk homotopy type (or shape)
of the Lorenz attractor is that of the branched manifold, which turns out to
be that of the figure eight. This agrees with the results found in \cite{Sanjhopf}. A
consequence of this is that the cohomology of the Lorenz attractor and its
cohomological Conley indices are isomorphic to $\mathbb{Z}$ in dimension zero, to $\mathbb{Z}\oplus\mathbb{Z}$ in dimension one and zero otherwise.

Our previous Theorem~\ref{attreppthm} implies that the 2-dimensional cohomological (Conley)
complexity of the travelling repeller $R_{r}$ is higher than that of $%
K_{24.06}$. As a matter of fact, $K_{24.06}$ has the cohomological Conley
indices of the (pointed) circle and hence $CH^{2}(K_{24.06})=\mathbb{\{}0%
\mathbb{\}}$, as it has been remarked in Section~\ref{sec:preturbulence}.  On the other hand, it
follows immediately from our next result that $CH^{2}(R_{r})=\mathbb{Z}%
\oplus \mathbb{Z}$.

The following theorem addresses a more general situation. We simplify the
hypotheses slightly to make the exposition simpler.

\begin{theorem}\label{travel}
Let $\varphi _{\lambda }:\mathbb{R}^{n}\times \mathbb{R}\rightarrow \mathbb{R%
}^{n}$, $\lambda \in \mathbb{R}$, be a continuous family of flows and let $%
\Omega $ be a global attractor for all the flows $\varphi _{\lambda }$. 
Suppose that $K$ and $C$ are isolated invariant sets for every $\lambda $
and that $(K,C)$ is a repeller-attractor decomposition of $\Omega $ for $%
\varphi _{\lambda _{0}}$ and $(K,C)$ is an attractor-repeller decomposition
of $\Omega $ for $\varphi _{\lambda _{1}}$, where $\lambda _{0}<\lambda
_{1}. $ Suppose, additionally, that the isolated invariant set $R_{\lambda }$
is a repeller of $\ \varphi _{\lambda }$ for $\lambda _{0}<\lambda <\lambda
_{1}$ and that $(K\cup C,R_{\lambda })$ is an attractor-repeller
decomposition of $\Omega $. Denote by $r_{k}$ the rank of $H^{k}(K)$ and by $%
r_{k}^{\prime }$ the rank of $H^{k}(C)$. Then we have the following Morse
equations

\[
r_{0}^{\prime }+(r_{1}^{\prime }+r_{0}^{\prime }-1)t+\sum_{k\geq
2}(r_{k}^{\prime }+r_{k-1}^{\prime })t^{k}=1+(1+t)Q_{1}(t),
\]

for the repeller-attractor
decomposition $(K,C)$ of $\varphi _{\lambda _{0}}|_{\Omega} $, 
\[
r_{0}+r_{0}^{\prime }+(r_{1}+r_{0}+r_{1}^{\prime }+r_{0}^{\prime
}-1)t+\sum_{k\geq 2}(r_{k}+r_{k-1}+r_{k}^{\prime }+r_{k-1}^{\prime
})t^{k}=1+(1+t)Q_{2}(t),
\]
 for the attractor-repeller decomposition $(K\cup
C,R_{\lambda })$ of $\varphi _{\lambda }|_{\Omega} $ with $\lambda _{0}<\lambda
<\lambda _{1}$ and

\[
r_{0}+(r_{1}+r_{0}-1)t+\sum_{k\geq
2}(r_{k}+r_{k-1})t^{k}=1+(1+t)Q_{3}(t),
\]
for the
attractor-repeller decomposition $(K,C)$ of $\varphi _{\lambda _{1}}|_{\Omega} $.
\end{theorem}

To prove Theorem~\ref{travel} we shall make use of the following lemma.

\begin{lemma}\label{lem:travel}
In the conditions of Theorem~\ref{travel} we have that
\begin{enumerate}
\item[a)] $CH^{k}(C)\cong H^{k}(C)\cong CH^{k+1}(K)$ if $k>0$ and $CH^0(C)\cong H^0(C)\cong\mathbb{Z}\oplus CH^1(K)$ for the flow $%
\varphi _{\lambda _{0}}$. 

\item[b)] $CH^{k}(K)\cong H^{k}(K)\cong CH^{k+1}(C)$ if $k>0$ and $CH^0(K)\cong H^0(K)\cong\mathbb{Z}\oplus CH^1(C)$ for the flow $%
\varphi _{\lambda _{1}}$.

\item[c)] $CH^{k+1}(R_{\lambda })\cong CH^{k}(K\cup C)\cong H^{k}(K\cup C)$ if $k>0$ and $CH^0(K\cup C)\cong H^0(K\cup C)\cong\mathbb{Z}\oplus CH^1(R_\lambda)$ for the flow $\varphi _{\lambda }$with $\lambda _{0}<\lambda <\lambda _{1}.$

\end{enumerate}
\end{lemma}

\begin{proof}

We shall prove a more general result which encompasses  a), b) and c). Let $\varphi:\mathbb{R}^n\times\mathbb{R}\to \mathbb{R}^n$ a dissipative flow with global attractor $\Omega$. Suppose that $(A,R)$ is an attractor repeller decomposition of $\Omega$ and consider the cohomology long exact sequence associated to the decomposition $(A,R)$,  
\[
\cdots\xrightarrow{\partial} CH^*(R)\xrightarrow{j^*} CH^*(\Omega)\xrightarrow{i^*} CH^*(A)\xrightarrow{\partial}\cdots
\]
since $\Omega$ is a global attractor, it follows that $CH^k(\Omega)$ is $\mathbb{Z}$ if $k=0$ and zero if $k>0$. Taking this into account in the exact sequence it readily follows that $CH^k(A)\cong CH^{k+1}(R)$ if $k>0$. On the other hand, since none of the components $R$ is an attractor, it follows that $CH^0(R)=0$ and, hence, the initial part of the sequence looks like
\[
0\to CH^0(\Omega)\to CH^0(A)\to CH^1(R)\to 0
\]
The Universal Coefficient ensures that $CH^1(R)$ must be torsion free and, as a consequence, the short exact sequence splits. Then
\[
CH^0(A)\cong CH^0(\Omega)\oplus CH^1(R)\cong \mathbb{Z}\oplus CH^1(R)
\]
Notice that, since $A$ is an attractor for the flow $\varphi$ restricted to the global attractor $\Omega$, then $A$ is an attractor for $\varphi$. Therefore $CH^*(A)\cong H^*(A)$. The result follows by replacing $A$ and $R$ by the corresponding sets.
\end{proof}

\begin{proof}[Proof of Theorem~\ref{travel}]
The proof follows from Lemma~\ref{lem:travel} combined with the fact that $CH^k(\Omega)$ is $\mathbb{Z}$ if $k=0$ and zero if $k>0$, $\Omega$ being  a global attractor for each $\lambda$. 
\end{proof}

Concerning the previous lemma, it is interesting to note that, when $%
\Omega $ is a global attractor, then the topological properties of $K$ and $%
C $ determine the cohomological Conley indices and the Morse equations of
all the involved isolated invariant sets, including $R_{\lambda }.$ It is
also interesting to see how the transition from repeller-attractor to
attractor-repeller is reflected in the Morse equations.

Another situation, not applicable to the Lorenz equations but provided of
theoretical interest, is when we have a flow in a compact manifold $M$ and a
similar transition for a pair $(K,L).$ Then McCord duality for
attractor-repeller pairs \cite{McH,MrSr} is applicable and the equations are determined
by the topology of $K$ and $M$ alone (the Conley index properties of $C$
being dual to those of $M$).

We finally point out that the evolution of the Lorenz attractor that we have
just studied has a nice counterpart from the analytical point of view. The following statement summarizes the situation.

\subsection*{3} \textbf{The transition from the repeller-attractor decomposition }$%
(K_{24.06},\{C_{1},C_{2}\})$ \textbf{(\textit{creation of the Lorenz
attractor})} \textbf{to the attractor-repeller decomposition }$%
(K_{24.74},\{C_{1},\break 
C_{2}\})$ \textbf{(\textit{Hopf bifurcation})} \textbf{%
through the decomposition }$(K_r\cup \{C_{1},C_{2}\},R_{r}=S\cup L)$ \textbf{(%
\textit{involving the travelling repellers} }$R_{r}$\textbf{) of the global
attractor }$\Omega $ \textbf{is reflected in the Morse equations shown in Theorem~\ref{travel}}.

Applying Theorem~\ref{travel} to this situation we get that, for $r=24.06$ the Morse equations associated to the repeller-attractor decomposition $(K_{24.06},\{C_{1},C_{2}\})$ of $\varphi_{24.06}|_{\Omega_{24.06}}$ are
\[
2+t=1+(1+t),
\]
for $r$ with $24.06<r<24.74$ the Morse equations associated to the attractor-repeller decomposition $(K_r\cup \{C_{1},C_{2}\},R_{r}=S\cup L)$ of $\varphi_r|_{\Omega_r}$ are
\[
3+4t+2t^2=1+(1+t)(2+2t),
\]
and, for $r=24.74$ the Morse equations associated to the attractor-repeller decomposition $(K_{24.74},\{C_{1},C_{2}\})$ of $\varphi_{24.74}|_{\Omega_{24.74}}$ are
\[
1+2t+2t^2=1+(1+t)2t.
\]

\section*{Acknowledgements}

Some of the results of this paper were obtained while the second author was
visiting the University of Manchester. He wishes to express his gratitude to
the Department of Mathematics of this University and, very specially, to
Nigel Ray for his hospitality and friendship along many years. The authors
are also grateful to Jos\'{e} M. Montesinos-Amilibia and Jaime S\'{a}%
nchez-Gabites for many inspiring conversations. They also would like to thank the referees, whose useful comments and suggestions have improved the quality of the present manuscript.

\bibliographystyle{plain}  
\bibliography{Biblio1}  

\end{document}